\documentclass[12pt]{article}
\usepackage{amssymb}
\usepackage{amsfonts}
\usepackage{german}
\long\def\ig#1{\relax}
\ig{Thanks to Roberto Minio for this def'n.  Compare the def'n of
\comment in AMSTeX.}

\newcount \coefa
\newcount \coefb
\newcount \coefc
\newcount\tempcounta
\newcount\tempcountb
\newcount\tempcountc
\newcount\tempcountd
\newcount\xext
\newcount\yext
\newcount\xoff
\newcount\yoff
\newcount\gap%
\newcount\arrowtypea
\newcount\arrowtypeb
\newcount\arrowtypec
\newcount\arrowtyped
\newcount\arrowtypee
\newcount\height
\newcount\width
\newcount\xpos
\newcount\ypos
\newcount\run
\newcount\rise
\newcount\arrowlength
\newcount\halflength
\newcount\arrowtype
\newdimen\tempdimen
\newdimen\xlen
\newdimen\ylen
\newsavebox{\tempboxa}%
\newsavebox{\tempboxb}%
\newsavebox{\tempboxc}%

\makeatletter
\def\settypes(#1,#2,#3){\arrowtypea#1 \arrowtypeb#2 \arrowtypec#3}
\def\settoheight#1#2{\setbox\@tempboxa\hbox{#2}#1\ht\@tempboxa\relax}%
\def\settodepth#1#2{\setbox\@tempboxa\hbox{#2}#1\dp\@tempboxa\relax}%
\def\settokens[#1`#2`#3`#4]{%
     \def\tokena{#1}\def\tokenb{#2}\def\tokenc{#3}\def\tokend{#4}}
\def\setsqparms[#1`#2`#3`#4;#5`#6]{%
\arrowtypea #1
\arrowtypeb #2
\arrowtypec #3
\arrowtyped #4
\width #5
\height #6
}
\def\setpos(#1,#2){\xpos=#1 \ypos#2}

\def\bfig{\begin{picture}(\xext,\yext)(\xoff,\yoff)}
\def\efig{\end{picture}}

\def\putbox(#1,#2)#3{\put(#1,#2){\makebox(0,0){$#3$}}}

\def\settriparms[#1`#2`#3;#4]{\settripairparms[#1`#2`#3`1`1;#4]}%

\def\settripairparms[#1`#2`#3`#4`#5;#6]{%
\arrowtypea #1
\arrowtypeb #2
\arrowtypec #3
\arrowtyped #4
\arrowtypee #5
\width #6
\height #6
}

\def\resetparms{\settripairparms[1`1`1`1`1;500]\width 500}

\resetparms

\def\mvector(#1,#2)#3{
\put(0,0){\vector(#1,#2){#3}}%
\put(0,0){\vector(#1,#2){30}}%
}
\def\evector(#1,#2)#3{{
\arrowlength #3
\put(0,0){\vector(#1,#2){\arrowlength}}%
\advance \arrowlength by-30
\put(0,0){\vector(#1,#2){\arrowlength}}%
}}

\def\horsize#1#2{%
\settowidth{\tempdimen}{$#2$}%
#1=\tempdimen
\divide #1 by\unitlength
}

\def\vertsize#1#2{%
\settoheight{\tempdimen}{$#2$}%
#1=\tempdimen
\settodepth{\tempdimen}{$#2$}%
\advance #1 by\tempdimen
\divide #1 by\unitlength
}

\def\vertadjust[#1`#2`#3]{%
\vertsize{\tempcounta}{#1}%
\vertsize{\tempcountb}{#2}%
\ifnum \tempcounta<\tempcountb \tempcounta=\tempcountb \fi
\divide\tempcounta by2
\vertsize{\tempcountb}{#3}%
\ifnum \tempcountb>0 \advance \tempcountb by20 \fi
\ifnum \tempcounta<\tempcountb \tempcounta=\tempcountb \fi
}

\def\horadjust[#1`#2`#3]{%
\horsize{\tempcounta}{#1}%
\horsize{\tempcountb}{#2}%
\ifnum \tempcounta<\tempcountb \tempcounta=\tempcountb \fi
\divide\tempcounta by20
\horsize{\tempcountb}{#3}%
\ifnum \tempcountb>0 \advance \tempcountb by60 \fi
\ifnum \tempcounta<\tempcountb \tempcounta=\tempcountb \fi
}

\ig{ In this procedure, #1 is the paramater that sticks out all the way,
#2 sticks out the least and #3 is a label sticking out half way.  #4 is
the amount of the offset.}

\def\sladjust[#1`#2`#3]#4{%
\tempcountc=#4
\horsize{\tempcounta}{#1}%
\divide \tempcounta by2
\horsize{\tempcountb}{#2}%
\divide \tempcountb by2
\advance \tempcountb by-\tempcountc
\ifnum \tempcounta<\tempcountb \tempcounta=\tempcountb\fi
\divide \tempcountc by2
\horsize{\tempcountb}{#3}%
\advance \tempcountb by-\tempcountc
\ifnum \tempcountb>0 \advance \tempcountb by80\fi
\ifnum \tempcounta<\tempcountb \tempcounta=\tempcountb\fi
\advance\tempcounta by20
}

\def\putvector(#1,#2)(#3,#4)#5#6{{%
\xpos=#1
\ypos=#2
\run=#3
\rise=#4
\arrowlength=#5
\arrowtype=#6
\ifnum \arrowtype<0
    \ifnum \run=0
        \advance \ypos by-\arrowlength
    \else
        \tempcounta \arrowlength
        \multiply \tempcounta by\rise
        \divide \tempcounta by\run
        \ifnum\run>0
            \advance \xpos by\arrowlength
            \advance \ypos by\tempcounta
        \else
            \advance \xpos by-\arrowlength
            \advance \ypos by-\tempcounta
        \fi
    \fi
    \multiply \arrowtype by-1
    \multiply \rise by-1
    \multiply \run by-1
\fi
\ifnum \arrowtype=1
    \put(\xpos,\ypos){\vector(\run,\rise){\arrowlength}}%
\else\ifnum \arrowtype=2
    \put(\xpos,\ypos){\mvector(\run,\rise)\arrowlength}%
\else\ifnum\arrowtype=3
    \put(\xpos,\ypos){\evector(\run,\rise){\arrowlength}}%
\fi\fi\fi
}}

\def\putsplitvector(#1,#2)#3#4{
\xpos #1
\ypos #2
\arrowtype #4
\halflength #3
\arrowlength #3
\gap 140
\advance \halflength by-\gap
\divide \halflength by2
\ifnum \arrowtype=1
    \put(\xpos,\ypos){\line(0,-1){\halflength}}%
    \advance\ypos by-\halflength
    \advance\ypos by-\gap
    \put(\xpos,\ypos){\vector(0,-1){\halflength}}%
\else\ifnum \arrowtype=2
    \put(\xpos,\ypos){\line(0,-1)\halflength}%
    \put(\xpos,\ypos){\vector(0,-1)3}%
    \advance\ypos by-\halflength
    \advance\ypos by-\gap
    \put(\xpos,\ypos){\vector(0,-1){\halflength}}%
\else\ifnum\arrowtype=3
    \put(\xpos,\ypos){\line(0,-1)\halflength}%
    \advance\ypos by-\halflength
    \advance\ypos by-\gap
    \put(\xpos,\ypos){\evector(0,-1){\halflength}}%
\else\ifnum \arrowtype=-1
    \advance \ypos by-\arrowlength
    \put(\xpos,\ypos){\line(0,1){\halflength}}%
    \advance\ypos by\halflength
    \advance\ypos by\gap
    \put(\xpos,\ypos){\vector(0,1){\halflength}}%
\else\ifnum \arrowtype=-2
    \advance \ypos by-\arrowlength
    \put(\xpos,\ypos){\line(0,1)\halflength}%
    \put(\xpos,\ypos){\vector(0,1)3}%
    \advance\ypos by\halflength
    \advance\ypos by\gap
    \put(\xpos,\ypos){\vector(0,1){\halflength}}%
\else\ifnum\arrowtype=-3
    \advance \ypos by-\arrowlength
    \put(\xpos,\ypos){\line(0,1)\halflength}%
    \advance\ypos by\halflength
    \advance\ypos by\gap
    \put(\xpos,\ypos){\evector(0,1){\halflength}}%
\fi\fi\fi\fi\fi\fi
}

\def\putmorphism(#1)(#2,#3)[#4`#5`#6]#7#8#9{{%
\run #2
\rise #3
\ifnum\rise=0
  \puthmorphism(#1)[#4`#5`#6]{#7}{#8}{#9}%
\else\ifnum\run=0
  \putvmorphism(#1)[#4`#5`#6]{#7}{#8}{#9}%
\else
\setpos(#1)%
\arrowlength #7
\arrowtype #8
\ifnum\run=0
\else\ifnum\rise=0
\else
\ifnum\run>0
    \coefa=1
\else
   \coefa=-1
\fi
\ifnum\arrowtype>0
   \coefb=0
   \coefc=-1
\else
   \coefb=\coefa
   \coefc=1
   \arrowtype=-\arrowtype
\fi
\width=2
\multiply \width by\run
\divide \width by\rise
\ifnum \width<0  \width=-\width\fi
\advance\width by60
\if l#9 \width=-\width\fi
\putbox(\xpos,\ypos){#4}
{\multiply \coefa by\arrowlength
\advance\xpos by\coefa
\multiply \coefa by\rise
\divide \coefa by\run
\advance \ypos by\coefa
\putbox(\xpos,\ypos){#5} }%
{\multiply \coefa by\arrowlength
\divide \coefa by2
\advance \xpos by\coefa
\advance \xpos by\width
\multiply \coefa by\rise
\divide \coefa by\run
\advance \ypos by\coefa
\if l#9%
   \put(\xpos,\ypos){\makebox(0,0)[r]{$#6$}}%
\else\if r#9%
   \put(\xpos,\ypos){\makebox(0,0)[l]{$#6$}}%
\fi\fi }%
{\multiply \rise by-\coefc
\multiply \run by-\coefc
\multiply \coefb by\arrowlength
\advance \xpos by\coefb
\multiply \coefb by\rise
\divide \coefb by\run
\advance \ypos by\coefb
\multiply \coefc by70
\advance \ypos by\coefc
\multiply \coefc by\run
\divide \coefc by\rise
\advance \xpos by\coefc
\multiply \coefa by140
\multiply \coefa by\run
\divide \coefa by\rise
\advance \arrowlength by\coefa
\ifnum \arrowtype=1
   \put(\xpos,\ypos){\vector(\run,\rise){\arrowlength}}%
\else\ifnum\arrowtype=2
   \put(\xpos,\ypos){\mvector(\run,\rise){\arrowlength}}%
\else\ifnum\arrowtype=3
   \put(\xpos,\ypos){\evector(\run,\rise){\arrowlength}}%
\fi\fi\fi}\fi\fi\fi\fi}}

\def\puthmorphism(#1,#2)[#3`#4`#5]#6#7#8{{%
\xpos #1
\ypos #2
\width #6
\arrowlength #6
\putbox(\xpos,\ypos){#3\vphantom{#4}}%
{\advance \xpos by\arrowlength
\putbox(\xpos,\ypos){\vphantom{#3}#4}}%
\horsize{\tempcounta}{#3}%
\horsize{\tempcountb}{#4}%
\divide \tempcounta by2
\divide \tempcountb by2
\advance \tempcounta by30
\advance \tempcountb by30
\advance \xpos by\tempcounta
\advance \arrowlength by-\tempcounta
\advance \arrowlength by-\tempcountb
\putvector(\xpos,\ypos)(1,0){\arrowlength}{#7}%
\divide \arrowlength by2
\advance \xpos by\arrowlength
\vertsize{\tempcounta}{#5}%
\divide\tempcounta by2
\advance \tempcounta by20
\if a#8 %
   \advance \ypos by\tempcounta
   \putbox(\xpos,\ypos){#5}%
\else
   \advance \ypos by-\tempcounta
   \putbox(\xpos,\ypos){#5}%
\fi}}

\def\putvmorphism(#1,#2)[#3`#4`#5]#6#7#8{{%
\xpos #1
\ypos #2
\arrowlength #6
\arrowtype #7
\settowidth{\xlen}{$#5$}%
\putbox(\xpos,\ypos){#3}%
{\advance \ypos by-\arrowlength
\putbox(\xpos,\ypos){#4}}%
{\advance\arrowlength by-140
\advance \ypos by-70
\ifdim\xlen>0pt
   \if m#8%
      \putsplitvector(\xpos,\ypos){\arrowlength}{\arrowtype}%
   \else
      \putvector(\xpos,\ypos)(0,-1){\arrowlength}{\arrowtype}%
   \fi
\else
   \putvector(\xpos,\ypos)(0,-1){\arrowlength}{\arrowtype}%
\fi}%
\ifdim\xlen>0pt
   \divide \arrowlength by2
   \advance\ypos by-\arrowlength
   \if l#8%
      \advance \xpos by-40
      \put(\xpos,\ypos){\makebox(0,0)[r]{$#5$}}%
   \else\if r#8%
      \advance \xpos by40
      \put(\xpos,\ypos){\makebox(0,0)[l]{$#5$}}%
   \else
      \putbox(\xpos,\ypos){#5}%
   \fi\fi
\fi
}}

\def\topadjust[#1`#2`#3]{%
\yoff=10
\vertadjust[#1`#2`{#3}]%
\advance \yext by\tempcounta
\advance \yext by 10
}
\def\botadjust[#1`#2`#3]{%
\vertadjust[#1`#2`{#3}]%
\advance \yext by\tempcounta
\advance \yoff by-\tempcounta
}
\def\leftadjust[#1`#2`#3]{%
\xoff=0
\horadjust[#1`#2`{#3}]%
\advance \xext by\tempcounta
\advance \xoff by-\tempcounta
}
\def\rightadjust[#1`#2`#3]{%
\horadjust[#1`#2`{#3}]%
\advance \xext by\tempcounta
}
\def\rightsladjust[#1`#2`#3]{%
\sladjust[#1`#2`{#3}]{\width}%
\advance \xext by\tempcounta
}
\def\leftsladjust[#1`#2`#3]{%
\xoff=0
\sladjust[#1`#2`{#3}]{\width}%
\advance \xext by\tempcounta
\advance \xoff by-\tempcounta
}
\def\adjust[#1`#2;#3`#4;#5`#6;#7`#8]{%
\topadjust[#1``{#2}]
\leftadjust[#3``{#4}]
\rightadjust[#5``{#6}]
\botadjust[#7``{#8}]}

\def\putsquarep<#1>(#2)[#3;#4`#5`#6`#7]{{%
\setsqparms[#1]%
\setpos(#2)%
\settokens[#3]%
\puthmorphism(\xpos,\ypos)[\tokenc`\tokend`{#7}]{\width}{\arrowtyped}b%
\advance\ypos by \height
\puthmorphism(\xpos,\ypos)[\tokena`\tokenb`{#4}]{\width}{\arrowtypea}a%
\putvmorphism(\xpos,\ypos)[``{#5}]{\height}{\arrowtypeb}l%
\advance\xpos by \width
\putvmorphism(\xpos,\ypos)[``{#6}]{\height}{\arrowtypec}r%
}}

\def\putsquare{\@ifnextchar <{\putsquarep}{\putsquarep%
   <\arrowtypea`\arrowtypeb`\arrowtypec`\arrowtyped;\width`\height>}}
\def\square{\@ifnextchar< {\squarep}{\squarep
   <\arrowtypea`\arrowtypeb`\arrowtypec`\arrowtyped;\width`\height>}}
\def\squarep<#1>[#2`#3`#4`#5;#6`#7`#8`#9]{{
\setsqparms[#1]
\xext=\width                                          
\yext=\height                                         
\topadjust[#2`#3`{#6}]
\botadjust[#4`#5`{#9}]
\leftadjust[#2`#4`{#7}]
\rightadjust[#3`#5`{#8}]
\begin{picture}(\xext,\yext)(\xoff,\yoff)
\putsquarep<\arrowtypea`\arrowtypeb`\arrowtypec`\arrowtyped;\width`\height>%
(0,0)[#2`#3`#4`#5;#6`#7`#8`{#9}]%
\end{picture}%
}}

\def\putptrianglep<#1>(#2,#3)[#4`#5`#6;#7`#8`#9]{{%
\settriparms[#1]%
\xpos=#2 \ypos=#3
\advance\ypos by \height
\puthmorphism(\xpos,\ypos)[#4`#5`{#7}]{\height}{\arrowtypea}a%
\putvmorphism(\xpos,\ypos)[`#6`{#8}]{\height}{\arrowtypeb}l%
\advance\xpos by\height
\putmorphism(\xpos,\ypos)(-1,-1)[``{#9}]{\height}{\arrowtypec}r%
}}

\def\putptriangle{\@ifnextchar <{\putptrianglep}{\putptrianglep
   <\arrowtypea`\arrowtypeb`\arrowtypec;\height>}}
\def\ptriangle{\@ifnextchar <{\ptrianglep}{\ptrianglep
   <\arrowtypea`\arrowtypeb`\arrowtypec;\height>}}

\def\ptrianglep<#1>[#2`#3`#4;#5`#6`#7]{{
\settriparms[#1]%
\width=\height                         
\xext=\width                           
\yext=\width                           
\topadjust[#2`#3`{#5}]
\botadjust[#3``]
\leftadjust[#2`#4`{#6}]
\rightsladjust[#3`#4`{#7}]
\begin{picture}(\xext,\yext)(\xoff,\yoff)
\putptrianglep<\arrowtypea`\arrowtypeb`\arrowtypec;\height>%
(0,0)[#2`#3`#4;#5`#6`{#7}]%
\end{picture}%
}}

\def\putqtrianglep<#1>(#2,#3)[#4`#5`#6;#7`#8`#9]{{%
\settriparms[#1]%
\xpos=#2 \ypos=#3
\advance\ypos by\height
\puthmorphism(\xpos,\ypos)[#4`#5`{#7}]{\height}{\arrowtypea}a%
\putmorphism(\xpos,\ypos)(1,-1)[``{#8}]{\height}{\arrowtypeb}l%
\advance\xpos by\height
\putvmorphism(\xpos,\ypos)[`#6`{#9}]{\height}{\arrowtypec}r%
}}

\def\putqtriangle{\@ifnextchar <{\putqtrianglep}{\putqtrianglep
   <\arrowtypea`\arrowtypeb`\arrowtypec;\height>}}
\def\qtriangle{\@ifnextchar <{\qtrianglep}{\qtrianglep
   <\arrowtypea`\arrowtypeb`\arrowtypec;\height>}}

\def\qtrianglep<#1>[#2`#3`#4;#5`#6`#7]{{
\settriparms[#1]
\width=\height                         
\xext=\width                           
\yext=\height                          
\topadjust[#2`#3`{#5}]
\botadjust[#4``]
\leftsladjust[#2`#4`{#6}]
\rightadjust[#3`#4`{#7}]
\begin{picture}(\xext,\yext)(\xoff,\yoff)
\putqtrianglep<\arrowtypea`\arrowtypeb`\arrowtypec;\height>%
(0,0)[#2`#3`#4;#5`#6`{#7}]%
\end{picture}%
}}

\def\putdtrianglep<#1>(#2,#3)[#4`#5`#6;#7`#8`#9]{{%
\settriparms[#1]%
\xpos=#2 \ypos=#3
\puthmorphism(\xpos,\ypos)[#5`#6`{#9}]{\height}{\arrowtypec}b%
\advance\xpos by \height \advance\ypos by\height
\putmorphism(\xpos,\ypos)(-1,-1)[``{#7}]{\height}{\arrowtypea}l%
\putvmorphism(\xpos,\ypos)[#4``{#8}]{\height}{\arrowtypeb}r%
}}

\def\putdtriangle{\@ifnextchar <{\putdtrianglep}{\putdtrianglep
   <\arrowtypea`\arrowtypeb`\arrowtypec;\height>}}
\def\dtriangle{\@ifnextchar <{\dtrianglep}{\dtrianglep
   <\arrowtypea`\arrowtypeb`\arrowtypec;\height>}}

\def\dtrianglep<#1>[#2`#3`#4;#5`#6`#7]{{
\settriparms[#1]
\width=\height                         
\xext=\width                           
\yext=\height                          
\topadjust[#2``]
\botadjust[#3`#4`{#7}]
\leftsladjust[#3`#2`{#5}]
\rightadjust[#2`#4`{#6}]
\begin{picture}(\xext,\yext)(\xoff,\yoff)
\putdtrianglep<\arrowtypea`\arrowtypeb`\arrowtypec;\height>%
(0,0)[#2`#3`#4;#5`#6`{#7}]%
\end{picture}%
}}

\def\putbtrianglep<#1>(#2,#3)[#4`#5`#6;#7`#8`#9]{{%
\settriparms[#1]%
\xpos=#2 \ypos=#3
\puthmorphism(\xpos,\ypos)[#5`#6`{#9}]{\height}{\arrowtypec}b%
\advance\ypos by\height
\putmorphism(\xpos,\ypos)(1,-1)[``{#8}]{\height}{\arrowtypeb}r%
\putvmorphism(\xpos,\ypos)[#4``{#7}]{\height}{\arrowtypea}l%
}}

\def\putbtriangle{\@ifnextchar <{\putbtrianglep}{\putbtrianglep
   <\arrowtypea`\arrowtypeb`\arrowtypec;\height>}}
\def\btriangle{\@ifnextchar <{\btrianglep}{\btrianglep
   <\arrowtypea`\arrowtypeb`\arrowtypec;\height>}}

\def\btrianglep<#1>[#2`#3`#4;#5`#6`#7]{{
\settriparms[#1]
\width=\height                         
\xext=\width                           
\yext=\height                          
\topadjust[#2``]
\botadjust[#3`#4`{#7}]
\leftadjust[#2`#3`{#5}]
\rightsladjust[#4`#2`{#6}]
\begin{picture}(\xext,\yext)(\xoff,\yoff)
\putbtrianglep<\arrowtypea`\arrowtypeb`\arrowtypec;\height>%
(0,0)[#2`#3`#4;#5`#6`{#7}]%
\end{picture}%
}}

\def\putAtrianglep<#1>(#2,#3)[#4`#5`#6;#7`#8`#9]{{%
\settriparms[#1]%
\xpos=#2 \ypos=#3
{\multiply \height by2
\puthmorphism(\xpos,\ypos)[#5`#6`{#9}]{\height}{\arrowtypec}b}%
\advance\xpos by\height \advance\ypos by\height
\putmorphism(\xpos,\ypos)(-1,-1)[#4``{#7}]{\height}{\arrowtypea}l%
\putmorphism(\xpos,\ypos)(1,-1)[``{#8}]{\height}{\arrowtypeb}r%
}}

\def\putAtriangle{\@ifnextchar <{\putAtrianglep}{\putAtrianglep
   <\arrowtypea`\arrowtypeb`\arrowtypec;\height>}}
\def\Atriangle{\@ifnextchar <{\Atrianglep}{\Atrianglep
   <\arrowtypea`\arrowtypeb`\arrowtypec;\height>}}

\def\Atrianglep<#1>[#2`#3`#4;#5`#6`#7]{{
\settriparms[#1]
\width=\height                         
\xext=\width                           
\yext=\height                          
\topadjust[#2``]
\botadjust[#3`#4`{#7}]
\multiply \xext by2 
\leftsladjust[#3`#2`{#5}]
\rightsladjust[#4`#2`{#6}]
\begin{picture}(\xext,\yext)(\xoff,\yoff)%
\putAtrianglep<\arrowtypea`\arrowtypeb`\arrowtypec;\height>%
(0,0)[#2`#3`#4;#5`#6`{#7}]%
\end{picture}%
}}

\def\putAtrianglepairp<#1>(#2)[#3;#4`#5`#6`#7`#8]{{
\settripairparms[#1]%
\setpos(#2)%
\settokens[#3]%
\puthmorphism(\xpos,\ypos)[\tokenb`\tokenc`{#7}]{\height}{\arrowtyped}b%
\advance\xpos by\height
\advance\ypos by\height
\putmorphism(\xpos,\ypos)(-1,-1)[\tokena``{#4}]{\height}{\arrowtypea}l%
\putvmorphism(\xpos,\ypos)[``{#5}]{\height}{\arrowtypeb}m%
\putmorphism(\xpos,\ypos)(1,-1)[``{#6}]{\height}{\arrowtypec}r%
}}

\def\putAtrianglepair{\@ifnextchar <{\putAtrianglepairp}{\putAtrianglepairp%
   <\arrowtypea`\arrowtypeb`\arrowtypec`\arrowtyped`\arrowtypee;\height>}}
\def\Atrianglepair{\@ifnextchar <{\Atrianglepairp}{\Atrianglepairp%
   <\arrowtypea`\arrowtypeb`\arrowtypec`\arrowtyped`\arrowtypee;\height>}}

\def\Atrianglepairp<#1>[#2;#3`#4`#5`#6`#7]{{%
\settripairparms[#1]%
\settokens[#2]%
\width=\height
\xext=\width
\yext=\height
\topadjust[\tokena``]%
\vertadjust[\tokenb`\tokenc`{#6}]
\tempcountd=\tempcounta                       
\vertadjust[\tokenc`\tokend`{#7}]
\ifnum\tempcounta<\tempcountd                 
\tempcounta=\tempcountd\fi                    
\advance \yext by\tempcounta                  
\advance \yoff by-\tempcounta                 %
\multiply \xext by2 
\leftsladjust[\tokenb`\tokena`{#3}]
\rightsladjust[\tokend`\tokena`{#5}]%
\begin{picture}(\xext,\yext)(\xoff,\yoff)%
\putAtrianglepairp
<\arrowtypea`\arrowtypeb`\arrowtypec`\arrowtyped`\arrowtypee;\height>%
(0,0)[#2;#3`#4`#5`#6`{#7}]%
\end{picture}%
}}

\def\putVtrianglep<#1>(#2,#3)[#4`#5`#6;#7`#8`#9]{{%
\settriparms[#1]%
\xpos=#2 \ypos=#3
\advance\ypos by\height
{\multiply\height by2
\puthmorphism(\xpos,\ypos)[#4`#5`{#7}]{\height}{\arrowtypea}a}%
\putmorphism(\xpos,\ypos)(1,-1)[`#6`{#8}]{\height}{\arrowtypeb}l%
\advance\xpos by\height
\advance\xpos by\height
\putmorphism(\xpos,\ypos)(-1,-1)[``{#9}]{\height}{\arrowtypec}r%
}}

\def\putVtriangle{\@ifnextchar <{\putVtrianglep}{\putVtrianglep
   <\arrowtypea`\arrowtypeb`\arrowtypec;\height>}}
\def\Vtriangle{\@ifnextchar <{\Vtrianglep}{\Vtrianglep
   <\arrowtypea`\arrowtypeb`\arrowtypec;\height>}}

\def\Vtrianglep<#1>[#2`#3`#4;#5`#6`#7]{{
\settriparms[#1]
\width=\height                         
\xext=\width                           
\yext=\height                          
\topadjust[#2`#3`{#5}]
\botadjust[#4``]
\multiply \xext by2 
\leftsladjust[#2`#3`{#6}]
\rightsladjust[#3`#4`{#7}]
\begin{picture}(\xext,\yext)(\xoff,\yoff)%
\putVtrianglep<\arrowtypea`\arrowtypeb`\arrowtypec;\height>%
(0,0)[#2`#3`#4;#5`#6`{#7}]%
\end{picture}%
}}

\def\putVtrianglepairp<#1>(#2)[#3;#4`#5`#6`#7`#8]{{
\settripairparms[#1]%
\setpos(#2)%
\settokens[#3]%
\advance\ypos by\height
\putmorphism(\xpos,\ypos)(1,-1)[`\tokend`{#6}]{\height}{\arrowtypec}l%
\puthmorphism(\xpos,\ypos)[\tokena`\tokenb`{#4}]{\height}{\arrowtypea}a%
\advance\xpos by\height
\putvmorphism(\xpos,\ypos)[``{#7}]{\height}{\arrowtyped}m%
\advance\xpos by\height
\putmorphism(\xpos,\ypos)(-1,-1)[``{#8}]{\height}{\arrowtypee}r%
}}

\def\putVtrianglepair{\@ifnextchar <{\putVtrianglepairp}{\putVtrianglepairp%
    <\arrowtypea`\arrowtypeb`\arrowtypec`\arrowtyped`\arrowtypee;\height>}}
\def\Vtrianglepair{\@ifnextchar <{\Vtrianglepairp}{\Vtrianglepairp%
    <\arrowtypea`\arrowtypeb`\arrowtypec`\arrowtyped`\arrowtypee;\height>}}

\def\Vtrianglepairp<#1>[#2;#3`#4`#5`#6`#7]{{%
\settripairparms[#1]%
\settokens[#2]
\xext=\height                  
\width=\height                 
\yext=\height                  
\vertadjust[\tokena`\tokenb`{#4}]
\tempcountd=\tempcounta        
\vertadjust[\tokenb`\tokenc`{#5}]
\ifnum\tempcounta<\tempcountd%
\tempcounta=\tempcountd\fi
\advance \yext by\tempcounta
\botadjust[\tokend``]%
\multiply \xext by2
\leftsladjust[\tokena`\tokend`{#6}]%
\rightsladjust[\tokenc`\tokend`{#7}]%
\begin{picture}(\xext,\yext)(\xoff,\yoff)%
\putVtrianglepairp
<\arrowtypea`\arrowtypeb`\arrowtypec`\arrowtyped`\arrowtypee;\height>%
(0,0)[#2;#3`#4`#5`#6`{#7}]%
\end{picture}%
}}

\def\putCtrianglep<#1>(#2,#3)[#4`#5`#6;#7`#8`#9]{{%
\settriparms[#1]%
\xpos=#2 \ypos=#3
\advance\ypos by\height
\putmorphism(\xpos,\ypos)(1,-1)[``{#9}]{\height}{\arrowtypec}l%
\advance\xpos by\height
\advance\ypos by\height
\putmorphism(\xpos,\ypos)(-1,-1)[#4`#5`{#7}]{\height}{\arrowtypea}l%
{\multiply\height by 2
\putvmorphism(\xpos,\ypos)[`#6`{#8}]{\height}{\arrowtypeb}r}%
}}

\def\putCtriangle{\@ifnextchar <{\putCtrianglep}{\putCtrianglep
    <\arrowtypea`\arrowtypeb`\arrowtypec;\height>}}
\def\Ctriangle{\@ifnextchar <{\Ctrianglep}{\Ctrianglep
    <\arrowtypea`\arrowtypeb`\arrowtypec;\height>}}

\def\Ctrianglep<#1>[#2`#3`#4;#5`#6`#7]{{
\settriparms[#1]
\width=\height                          
\xext=\width                            
\yext=\height                           
\multiply \yext by2 
\topadjust[#2``]
\botadjust[#4``]
\sladjust[#3`#2`{#5}]{\width}
\tempcountd=\tempcounta                 
\sladjust[#3`#4`{#7}]{\width}
\ifnum \tempcounta<\tempcountd          
\tempcounta=\tempcountd\fi              
\advance \xext by\tempcounta            
\advance \xoff by-\tempcounta           %
\rightadjust[#2`#4`{#6}]
\begin{picture}(\xext,\yext)(\xoff,\yoff)%
\putCtrianglep<\arrowtypea`\arrowtypeb`\arrowtypec;\height>%
(0,0)[#2`#3`#4;#5`#6`{#7}]%
\end{picture}%
}}

\def\putDtrianglep<#1>(#2,#3)[#4`#5`#6;#7`#8`#9]{{%
\settriparms[#1]%
\xpos=#2 \ypos=#3
\advance\xpos by\height \advance\ypos by\height
\putmorphism(\xpos,\ypos)(-1,-1)[``{#9}]{\height}{\arrowtypec}r%
\advance\xpos by-\height \advance\ypos by\height
\putmorphism(\xpos,\ypos)(1,-1)[`#5`{#8}]{\height}{\arrowtypeb}r%
{\multiply\height by 2
\putvmorphism(\xpos,\ypos)[#4`#6`{#7}]{\height}{\arrowtypea}l}%
}}

\def\putDtriangle{\@ifnextchar <{\putDtrianglep}{\putDtrianglep
    <\arrowtypea`\arrowtypeb`\arrowtypec;\height>}}
\def\Dtriangle{\@ifnextchar <{\Dtrianglep}{\Dtrianglep
   <\arrowtypea`\arrowtypeb`\arrowtypec;\height>}}

\def\Dtrianglep<#1>[#2`#3`#4;#5`#6`#7]{{
\settriparms[#1]
\width=\height                         
\xext=\height                          
\yext=\height                          
\multiply \yext by2 
\topadjust[#2``]
\botadjust[#4``]
\leftadjust[#2`#4`{#5}]
\sladjust[#3`#2`{#5}]{\height}
\tempcountd=\tempcountd                
\sladjust[#3`#4`{#7}]{\height}
\ifnum \tempcounta<\tempcountd         
\tempcounta=\tempcountd\fi             
\advance \xext by\tempcounta           %
\begin{picture}(\xext,\yext)(\xoff,\yoff)
\putDtrianglep<\arrowtypea`\arrowtypeb`\arrowtypec;\height>%
(0,0)[#2`#3`#4;#5`#6`{#7}]%
\end{picture}%
}}

\def\setrecparms[#1`#2]{\width=#1 \height=#2}%
\def\recursep<#1`#2>[#3;#4`#5`#6`#7`#8]{{%
\width=#1 \height=#2
\settokens[#3]
\settowidth{\tempdimen}{$\tokena$}
\ifdim\tempdimen=0pt
  \savebox{\tempboxa}{\hbox{$\tokenb$}}%
  \savebox{\tempboxb}{\hbox{$\tokend$}}%
  \savebox{\tempboxc}{\hbox{$#6$}}%
\else
  \savebox{\tempboxa}{\hbox{$\hbox{$\tokena$}\times\hbox{$\tokenb$}$}}%
  \savebox{\tempboxb}{\hbox{$\hbox{$\tokena$}\times\hbox{$\tokend$}$}}%
  \savebox{\tempboxc}{\hbox{$\hbox{$\tokena$}\times\hbox{$#6$}$}}%
\fi
\ypos=\height
\divide\ypos by 2
\xpos=\ypos
\advance\xpos by \width
\xext=\xpos \yext=\height
\topadjust[#3`\usebox{\tempboxa}`{#4}]%
\botadjust[#5`\usebox{\tempboxb}`{#8}]%
\sladjust[\tokenc`\tokenb`{#5}]{\ypos}%
\tempcountd=\tempcounta
\sladjust[\tokenc`\tokend`{#5}]{\ypos}%
\ifnum \tempcounta<\tempcountd
\tempcounta=\tempcountd\fi
\advance \xext by\tempcounta
\advance \xoff by-\tempcounta
\rightadjust[\usebox{\tempboxa}`\usebox{\tempboxb}`\usebox{\tempboxc}]%
\bfig
\putCtrianglep<-1`1`1;\ypos>(0,0)[`\tokenc`;#5`#6`{#7}]%
\puthmorphism(\ypos,0)[\tokend`\usebox{\tempboxb}`{#8}]{\width}{-1}b%
\puthmorphism(\ypos,\height)[\tokenb`\usebox{\tempboxa}`{#4}]{\width}{-1}a%
\advance\ypos by \width
\putvmorphism(\ypos,\height)[``\usebox{\tempboxc}]{\height}1r%
\efig
}}

\def\recurse{\@ifnextchar <{\recursep}{\recursep<\width`\height>}}

\def\puttwohmorphisms(#1,#2)[#3`#4;#5`#6]#7#8#9{{%
\puthmorphism(#1,#2)[#3`#4`]{#7}0a
\ypos=#2
\advance\ypos by 20
\puthmorphism(#1,\ypos)[\phantom{#3}`\phantom{#4}`#5]{#7}{#8}a
\advance\ypos by -40
\puthmorphism(#1,\ypos)[\phantom{#3}`\phantom{#4}`#6]{#7}{#9}b
}}

\def\puttwovmorphisms(#1,#2)[#3`#4;#5`#6]#7#8#9{{%
\putvmorphism(#1,#2)[#3`#4`]{#7}0a
\xpos=#1
\advance\xpos by -20
\putvmorphism(\xpos,#2)[\phantom{#3}`\phantom{#4}`#5]{#7}{#8}l
\advance\xpos by 40
\putvmorphism(\xpos,#2)[\phantom{#3}`\phantom{#4}`#6]{#7}{#9}r
}}

\def\puthcoequalizer(#1)[#2`#3`#4;#5`#6`#7]#8#9{{%
\setpos(#1)%
\puttwohmorphisms(\xpos,\ypos)[#2`#3;#5`#6]{#8}11%
\advance\xpos by #8
\puthmorphism(\xpos,\ypos)[\phantom{#3}`#4`#7]{#8}1{#9}
}}

\def\putvcoequalizer(#1)[#2`#3`#4;#5`#6`#7]#8#9{{%
\setpos(#1)%
\puttwovmorphisms(\xpos,\ypos)[#2`#3;#5`#6]{#8}11%
\advance\ypos by -#8
\putvmorphism(\xpos,\ypos)[\phantom{#3}`#4`#7]{#8}1{#9}
}}

\def\putthreehmorphisms(#1)[#2`#3;#4`#5`#6]#7(#8)#9{{%
\setpos(#1) \settypes(#8)
\if a#9 %
     \vertsize{\tempcounta}{#5}%
     \vertsize{\tempcountb}{#6}%
     \ifnum \tempcounta<\tempcountb \tempcounta=\tempcountb \fi
\else
     \vertsize{\tempcounta}{#4}%
     \vertsize{\tempcountb}{#5}%
     \ifnum \tempcounta<\tempcountb \tempcounta=\tempcountb \fi
\fi
\advance \tempcounta by 60
\puthmorphism(\xpos,\ypos)[#2`#3`#5]{#7}{\arrowtypeb}{#9}
\advance\ypos by \tempcounta
\puthmorphism(\xpos,\ypos)[\phantom{#2}`\phantom{#3}`#4]{#7}{\arrowtypea}{#9}
\advance\ypos by -\tempcounta \advance\ypos by -\tempcounta
\puthmorphism(\xpos,\ypos)[\phantom{#2}`\phantom{#3}`#6]{#7}{\arrowtypec}{#9}
}}

\def\putarc(#1,#2)[#3`#4`#5]#6#7#8{{%
\xpos #1
\ypos #2
\width #6
\arrowlength #6
\putbox(\xpos,\ypos){#3\vphantom{#4}}%
{\advance \xpos by\arrowlength
\putbox(\xpos,\ypos){\vphantom{#3}#4}}%
\horsize{\tempcounta}{#3}%
\horsize{\tempcountb}{#4}%
\divide \tempcounta by2
\divide \tempcountb by2
\advance \tempcounta by30
\advance \tempcountb by30
\advance \xpos by\tempcounta
\advance \arrowlength by-\tempcounta
\advance \arrowlength by-\tempcountb
\halflength=\arrowlength \divide\halflength by 2
\divide\arrowlength by 5
\put(\xpos,\ypos){\bezier{\arrowlength}(0,0)(50,50)(\halflength,50)}
\ifnum #7=-1 \put(\xpos,\ypos){\vector(-3,-2)0} \fi
\advance\xpos by \halflength
\put(\xpos,\ypos){\xpos=\halflength \advance\xpos by -50
   \bezier{\arrowlength}(0,50)(\xpos,50)(\halflength,0)}
\ifnum #7=1 {\advance \xpos by
   \halflength \put(\xpos,\ypos){\vector(3,-2)0}} \fi
\advance\ypos by 50
\vertsize{\tempcounta}{#5}%
\divide\tempcounta by2
\advance \tempcounta by20
\if a#8 %
   \advance \ypos by\tempcounta
   \putbox(\xpos,\ypos){#5}%
\else
   \advance \ypos by-\tempcounta
   \putbox(\xpos,\ypos){#5}%
\fi
}}

\makeatother
\newtheorem{theorem}{Satz}

\newtheorem{corollary}[theorem]{Folgerung}
\newtheorem{remark}[theorem]{Bemerkung}
\newtheorem{proposition}[theorem]{Proposition}

\newtheorem{lemma}[theorem]{Lemma}
\newtheorem{beispiel}[theorem]{Beispiel}

\newcommand{\supp}{\mbox{Supp }}

\newcommand{\ay}{{\mathfrak{a}}}
\newcommand{\py}{{\mathfrak{p}}}
\newcommand{\my}{{\mathfrak{m}}}

\newcommand{\by}{{\mathfrak{b}}}
\newcommand{\cy}{{\mathfrak{c}}}

\newenvironment{proof}{{\it Beweis}. $\;\;$}{\hspace*{\fill} $\Box$}

\begin{document}

\title{\Large\bf "Uber die assoziierten Primideale\\ der Vervollst"andigung }

\author{Helmut Z"oschinger\\ 
Mathematisches Institut der Universit"at M"unchen\\
Theresienstr. 39, D-80333 M"unchen, Germany\\
E-mail: zoeschinger@mathematik.uni-muenchen.de}

\date{}
\maketitle

\vspace{1cm}

\begin{center}  
{\bf Abstract}
\end{center}

Let $(R,\my)$ be a noetherian local ring and let $M$ be an $R$-module such that $\bigcap\limits_{n\geq 1} \my^n M=0.$ Let $\hat{M}$ be the completion of $M$. We show that Ass$(\hat{M})=$ Koatt$(M)$ holds in the following three cases:
if $\dim(R)\leq 1,$ if $\hat{M}$ as $R$-module is flat, or if $M$ is the direct sum of $R$-modules which are finitely generated.
If $M$ is pure in $\hat{M}$ then at least Ass$(\hat{M}) \subset $ Koatt$(M)$ holds. If the conjecture by A.-M.Simon on complete $R$-modules is valid then one has Koatt$(M)\subset $ Ass$(\hat{M}).$

\vspace{0.5cm}


\vspace{0.5cm}



\section*{Einleitung}

Sei $(R,\my)$ ein kommutativer noetherscher lokaler Ring, $M$ ein $R$-Modul mit $\bigcap\limits_{n\geq 1} \my^nM=0$ und $\hat{M}=\lim\limits_{\leftarrow} M/\my^nM$ die Vervollst"andigung von $M$ in der $\my$-adischen Topologie.
Wir interessieren uns in dieser Arbeit f"ur die Menge Ass$(\hat{M})$ aller zum $R$-Modul $\hat{M}$ assoziierten Primideale. Ist $M$ endlich erzeugt, gilt Ass$(\hat{M})=$ Ass$(M)$. Ist $M$ nicht endlich erzeugt, wei"s man "uber die Menge Ass$(\hat{M})$ sehr wenig.
Sie kann viel gr"o"ser als Ass$(M)$ sein, wie die Beispiele $M_1=\coprod\limits_{i=1}^\infty R/\my^i$ und $M_2=\coprod\limits_{i=1}^\infty (R/\my^i)^0$ im dritten Abschnitt zeigen (wobei $A^0= $ Hom$_R(A,E)$ das Matlis-Duale des $R$-Moduls $A$ sei):
Es ist Ass$(M_1)=$ Ass$(M_2)=\{\my\},$ w"ahrend nach (3.5) und (3.6) gilt
$$\mbox{Ass}(\hat{M}_1)\;=\;\{\my\}\,\cup\,\mbox{Ass}(R)\quad \mbox{und}\;\mbox{ Ass}(\hat{M}_2)\;=\;\mbox{Spec}(R).$$
Bei beliebigem $M$ scheint die Menge Koatt$(M)$ ein guter Kandidat f"ur Ass$(\hat{M})$ zu sein: $\py \in $ Spec$(R)$ hei"st {\it koattachiert} zu $M$, wenn es einen Untermodul $U$ von $M$ gibt mit $\py = $ Ann$_R(U).$
F"ur die rein-injektive H"ulle $N$ von $M$ zeigen wir nun  Koatt$(N) = $ Koatt$(M)$, und falls $M$ rein in $\hat{M}$ ist, entsprechend Koatt$(\hat{M})= $ Koatt$(M)$, woraus Ass$(\hat{M}) \subset $ Koatt$(M)$ folgt. 
Der Vergleich zwischen $\hat{M}$ und $N$ f"uhrt noch weiter, denn mit $N_1/M=P(N/M)$, dem gr"o"sten radikalvollen Untermodul von $N/M$, und $H(N_1)=\bigcap\limits_{n\geq 1} \my^n N_1$ gilt nach (1.6)
$$\hat{M} \cong\;N_1/H(N_1).$$
Daraus folgen die beiden Hauptergebnisse (1.7) und (1.8) des ersten Abschnittes:
Ist $\dim(R)\leq 1$ oder $\hat{M}$ als $R$-Modul flach, so gilt Ass$(\hat{M})=$ Koatt$(M)$.

Weil im allgemeinen $M$ nicht rein in $\hat{M}$ ist, untersuchen wir im zweiten Abschnitt sogenannte {\it totalseparierte} $R$-Moduln, bei denen nicht nur $M$, sondern alle $X\otimes_R M$, mit $X$ endlich erzeugt, separiert sind. Das ist
nach (2.1) genau dann der Fall, wenn die rein-injektive H"ulle von $M$ separiert ist. Zusammen mit dem Begriff des Mittag-Leffler-Moduls -- f"ur jede Familie
$Q_i(i \in I)$ von $R$-Moduln ist die kanonische Abbildung $M\otimes_R (\prod Q_i)\to \prod (M \otimes_R Q_i)\;$ injektiv -- erhalten wir die Implikationen
$$M\mbox{ ist Mittag-Leffler}\;\Longrightarrow\;M \mbox{ ist totalsepariert}\;\Longrightarrow \;M \mbox{ ist rein in } \hat{M},$$
und wann hier Umkehrungen gelten, wird in (2.4) bis (2.7) untersucht.
Anschlie"send geben wir neue Beispiele daf"ur, da"s $M$ nicht rein in $\hat{M}$ ist (2.8), ja sogar $\hat{M}$ flach und $M$ nicht flach ist (2.9).

Der dritte und letzte Abschnitt hat als Ziel, die Formel
$$\mbox{Ass }(\hat{M})\;=\;\mbox{Koatt }(M)$$
f"ur jeden $R$-Modul $M$ zu zeigen, der direkte Summe von endlich erzeugten (allgemeiner koatomaren) Moduln ist (3.4).
Ob diese Formel f"ur jeden separierten $R$-Modul $M$ gilt, ist ein ungel"ostes Problem und h"angt eng mit der von Simon 1990 gestellten Frage zusammen, ob $\supp (M)=\,V($Ann$_R(M))$ f"ur jeden {\it vollst"andigen} $R$-Modul $M$ gilt (3.7).

\section{"Uber die Inklusion Ass$(\hat{M})\subset $ Koatt$(M)$}

Die Menge Koatt$(M)$ hat gegen"uber den assoziierten Primidealen einige Vorteile: Ist $\py \in $ Spec$(R)$ Durchschnitt von beliebig vielen koattachierten Primidealen, ist auch $\py \in $ Koatt$(M)$, und ist $\py$ ein Primdivisor von Ann$_R(M)$, folgt schon $\py \in $ Koatt$(M)$. Falls $M$ vollst"andig ist, gilt noch mehr:

\begin{lemma}
Ist $M$ vollst"andig, so ist jedes $\py \in $ Koatt$(M)$ Durchschnitt von assoziierten Primidealen.
\end{lemma}

\begin{proof}
{\it 1. Schritt}\quad Es gilt $\bigcap $ Ass$(M)=\sqrt{\mbox{Ann}_R(M)}.$
Das folgt wie bei Simon \cite[p.243]{9} mit Hilfe des Satzes von Baire:
Aus $r \in \bigcap $ Ass$(M)$ folgt mit der Abk"urzung $M[\ay]=$ Ann$_M(\ay)$, da"s $M=\bigcup\limits_{n\geq 1} M[r^n]$ ist, also ein $M[r^m]$ einen inneren Punkt besitzt, d.h. $x+\my^e M\subset M[r^m]$ gilt f"ur ein $x \in M,\;
e\geq 1.$
Aus $r^m\cdot \my^eM=0$ folgt $r^{m+e}\cdot M=0,$ also $r \in \sqrt{\mbox{Ann}_R(M)}.$

{\it 2. Schritt}\quad Sei jetzt $\py \in $ Koatt$(M)$. Als abgeschlossener Untermodul ist auch $M_1=M[\py]$ vollst"andig, also nach dem 1. Schritt $\; \bigcap $ Ass$(M_1) = \sqrt{\mbox{Ann}_R(M_1)}=\py$ wie behauptet.
\end{proof}

\begin{corollary}
Ist $M$ vollst"andig, so gilt Ass$(M^{00})=$ Koatt$(M).$
\end{corollary}

\begin{proof}
Stets ist Ass$(M^{00}) \subset $ Koatt$(M^{00})=$ Koatt$(M)$, und bei vollst"andigem $M$ ist $\py \in $ Koatt$(M)$ von der Form $\py = \bigcap\limits_{\lambda \in \Lambda} \py_\lambda,$ alle $\py_\lambda \in $ Ass$(M)$, also nach \cite[Lemma 2.5]{13} $\py \in $ Ass$(M^{00}).$
\end{proof}

\begin{corollary}
Ist $M$ separiert, so ist Ass$(\hat{M}) \subset $ Koatt$(M)$ "aquivalent mit Koatt$(\hat{M})=$ Koatt$(M)$.
\end{corollary}

\begin{proof}
$''\Leftarrow''\quad$ ist klar, und bei $''\Rightarrow''\;$ ist jedes $\py \in $ Koatt$(\hat{M})$ von der Form $\py = \bigcap\limits_{\lambda \in \Lambda} \py_\lambda,$ alle $\py_\lambda \in $ Ass$(\hat{M}).$
Nach Voraussetzung ist dann jedes $\py_\lambda \in $ Koatt$(M)$, also auch $\py \in $ Koatt$(M).$
\end{proof}

Ein $R$-Modul $A$ hei"st {\it radikalvoll}, wenn $\my A=A$ gilt. Ist $A$ beliebig, sei $P(A)$ der gr"o"ste radikalvolle Untermodul von $A$, und falls $P(A)=0$ ist, hei"st $A$ {\it reduziert}.

\begin{lemma}
Sei $M$ vollst"andig und $U$ ein Untermodul von $M$. Genau dann ist $U$ vollst"andig, wenn $M/U$ reduziert ist.
\end{lemma}

\begin{proof}
$''\Rightarrow''\quad$ Mit $M_1/U =P(M/U)$ gilt $U+\my M_1=M_1$, wegen der Vollst"andigkeit von $U$ also nach Simon \cite[p.232, Lemma]{9} $U+H(M_1)=M_1$, wegen $H(M)=0$ also $U=M_1$.

$''\Leftarrow''\quad$ Nach Jensen \cite[Proposition 3]{5} ist ein $R$-Modul $A$ genau dann vollst"andig, wenn er kotorsion und separiert ist, d.h. wenn Ext$^1_R(C,A)=0$ ist f"ur alle flachen $R$-Moduln $C$ und $H(A)=0.$ Mit Hilfe eines Basis-Untermoduls $B$ von $C$ (siehe \cite[Theorem 7.10]{6} kann man $C$ flach und radikalvoll annehmen, und dann ist in der exakten Folge
$$\mbox{Hom}_R(C,M/U)\;\longrightarrow\;\mbox{Ext}_R^1(C,U)\;\longrightarrow\;\mbox{Ext}_R^1(C,M)$$
auch das erste Glied nach Voraussetzung Null, also $U$ kotorsion, $U$ vollst"andig wie behauptet.
\end{proof}

\begin{lemma}
Ist $M \subset N$ eine rein-injektive H"ulle von $M$, so gilt:
\begin{enumerate}
\item[(a)] Koatt$(N)\,=$ Koatt$(M).$
\item[(b)]
$ H(N)$ ist radikalvoll und rein-injektiv.
\item[(c)]
 $N/H(N)$ ist direkter Summand eines Produktes von Moduln endlicher L"ange.
\end{enumerate}
\end{lemma}

\begin{proof}
\begin{enumerate}
\item[(a)]
Die kanonische Abbildung $\alpha: M \to M^{00}$ l"a"st sich, weil $M\subset N$ rein-wesentlich ist, zu einem Monomorphismus $\beta: N \to M^{00}$ hochheben
und es folgt Koatt$(N) \subset $ Koatt$(M^{00})=$ Koatt$(M)$, w"ahrend $''\supset''$ klar ist.
\item[(b)]
{\it 1. Schritt} \quad Zu jedem $R$-Modul $N$ gibt es einen reinen Monomorphismus $N \to \prod\limits_{i \in I} A_i$, in dem alle $A_i$ artinsch sind.
Klar ist $N \to N^{00}$ ein reiner Monomorphismus, so da"s mit $D=N^0$ nur noch ein reiner Monomorphismus $D^0 \to \prod\limits_{i \in I} A_i$ anzugeben ist: Ist $\{B_i\;|i \in I\}\,$ die Menge aller endlich erzeugten Untermoduln von $D$, wird $\coprod\limits_{i \in I} B_i \twoheadrightarrow D$ ein reiner Epimorphismus, also $D^0 \to \prod\limits_{i \in I}(B_i)^0$ ein zerfallender Monomorphismus, und alle $(B_i)^0$ sind artinsch.

{\it 2. Schritt} \quad Ist jetzt wie in der Voraussetzung $N$ rein-injektiv, zerf"allt $N \to \prod\limits_{i \in I} A_i$, also auch $H(N)\to \prod\limits_{i \in I} H(A_i).$
Jedes der $H(A_i)$ ist aber radikalvoll und rein-injektiv, also auch ihr Produkt, also auch der direkte Summand $H(N)$.
\item[(c)]
Der zerfallende Monomorphismus $N \to \prod\limits_{i \in I} A_i$ induziert einen zerfallenden Monomorphismus $N/H(N) \to \prod\limits_{i \in I} A_i /H(A_i),$ in dem alle $A_i/H(A_i)$ artinsch und separiert, also von endlicher L"ange sind.
\end{enumerate}
\end{proof}

\begin{corollary}
Sei $M$ separiert, $M \subset N$ eine rein-injektive H"ulle von $M$ und $N_1/M=P(N/M).$
Dann ist $N_1$ kotorsion und
$$\hat{M}\;\cong\;N_1/H(N_1).$$
\end{corollary}

\begin{proof}
Nach (1.5c) ist $N/H(N)$ vollst"andig, und weil $H(N)$ nach (1.5b) radikalvoll, also in $N_1$ enthalten ist, gilt sogar $H(N)=H(N_1).$
Weil $N/N_1$ reduziert ist, ist nach (1.4) auch $N_1/H(N_1)$ vollst"andig, und weil wieder nach (1.5b) $H(N_1)$ kotorsion ist, ist das auch $N_1.$

Bleibt zu zeigen, da"s der kanonische Monomorphismus $\alpha: M \to N_1/H(N_1)$ auf den Vervollst"andigungen einen Isomorphismus $\hat{\alpha}$ induziert, d.h. in $\overline{N_1} = N_1/H(N_1)$ der Untermodul  $\overline{M} = $ Bild $ \alpha$ ein dichter Unterraum bez"uglich der $\my$-adischen Topologie ist:
Aus  $M+\my\cdot N_1=N_1$ folgt $\overline{M} + \my \cdot \overline{N_1}=\overline{N_1},$
so da"s nur noch $\overline{M} \cap \my^n\cdot \overline{N_1}=\my^n\cdot \overline{M}$ f"ur alle $n\geq 1$ zu zeigen ist, und das ist klar wegen $(M+H(N_1))\,\cap \my^n\cdot N_1\,=\my^n\cdot M+H(N_1).$
\end{proof}

\begin{theorem}
Ist $M$ separiert und $\dim (R) \leq 1,$ so gilt Ass$(\hat{M})=$ Koatt$(M).$
\end{theorem}

\begin{proof}
{\it 1. Schritt} \quad Allein aus $dim (R)\leq 1$ folgt f"ur jeden sockelfreien $R$-Modul $B$ und jeden radikalvollen Untermodul $A$, da"s auch $B/A$ sockelfrei ist:
Wegen $\my \not \subset \bigcup $ Ass$(A) \cup \, \bigcup $ Koass$(A)$ gibt es ein $r \in \my$, das bijektiv auf $A$ operiert, so da"s Ext$^1_R(k,A)=0$ ist und aus
Hom$_R(k,B) \to $ Hom$_R(k,B/A) \to 0$ die Behauptung folgt.

{\it 2. Schritt} \quad
F"ur $''\subset ''$ sei $\py \in $ Ass$(\hat{M})$: Falls $\py \neq \my$, ist $\py$ minimal "uber Ann$_R(\hat{M}) = $ Ann$_R(M),$ also $\py \in $ Koatt$(M).$
Falls $\py =\my,$ ist So$(\hat{M})\neq 0,$ mit den Bezeichnungen von (1.6) also So$(N_1/H(N_1))\neq 0.$
Nach dem ersten Schritt folgt So$(N_1)\neq 0,\;\,\my \in $ Koatt$(N_1)\subset $ Koatt$(N),$ also nach (1.5a) $ \my \in $ Koatt$(M).$

{\it 3. Schritt} \quad 
F"ur $''\supset''$ gilt ohne jede Beschr"ankung an $\dim(R):$ Ist $M$ separiert und Ass$(\hat{M})$ endlich, folgt Ass$(\hat{M}) = $ Koatt$(\hat{M})$.
Jedes $\py \in $ Koatt$(\hat{M})$ ist ja nach (1.1) von der Form $\py = \py_1 \cap \cdots \cap \py_m$ mit $\py_i \in $ Ass$(\hat{M}),$ und daraus folgt $\py \in $ Ass$(\hat{M}).$
\end{proof}

\begin{theorem}
Sei $M$ separiert und $\hat{M}$ als $R$-Modul flach. Mit einer rein-injektiven H"ulle $M \subset N$ gilt dann:
\begin{enumerate}
\item[(a)]
$N/M$ ist radikalvoll und $H(N) \subset^\oplus N.$
\item[(b)]
$\hat{M} \,\cong\, N/H(N).$
\item[(c)]
Ass$(\hat{M}) = $ Koatt$(M)$.
\end{enumerate}
\end{theorem}

\begin{proof}
\begin{enumerate}
\item[(a)]
Nach (1.6) ist $\hat{M} \cong N_1/H(N_1),$ also $H(N_1)$ rein in $N_1$, und weil $H(N_1)=H(N)$ nach (1.5b) rein-injektiv ist, folgt $H(N_1) \oplus X = N_1$.
Der vollst"andige flache $R$-Modul $X \cong \hat{M}$ ist nach Jensen \cite[Proposition 4]{5} rein-injektiv, so da"s auch $N_1$ rein-injektiv ist und aus der Minimalit"atseigenschaft der rein-injektiven H"ulle folgt $N_1=N.$
Damit sind beide Punkte bewiesen, ebenso (b).
\item[(c)]
$''\subset ''$ \quad Ist $\py$ assoziiert zu $\hat{M} \cong N/H(N),$ folgt wegen $H(N) \subset^\oplus N$ sogar $\py \in $ Ass$(N)$, also wieder nach (1.5a) $\py \in $ Koatt$(M)$.

$''\supset''$ \quad Wie im dritten Beweisschritt von (1.7), denn zu jedem flachen $R$-Modul $C$ gibt es einen reinen Epimorphismus $R^{(I)} \to C$, so da"s Ass$(C) \subset $ Ass$(R)$ endlich ist.
\end{enumerate}
\end{proof}

"Uber einem diskreten Bewertungsring ist nach Fuchs, Salce und Zanardo \cite[Lemma 5]{3} ein reiner Untermodul $A$ von $B$ genau dann rein-wesentlich in $B$, wenn $B/A$ teilbar und $A^1$ gro"s in $B^1$ ist. Erstaunlicherweise gilt die Implikation $''\Leftarrow''$ "uber jedem noetherschen lokalen Ring:

\begin{lemma}
Sei $A$ ein reiner Untermodul von $B$ und $B/A$ radikalvoll, $H(A)$ gro"s in $H(B).$
Dann ist $A$ sogar rein-wesentlich in $B$.
\end{lemma}

\begin{proof}
Sei $X \subset B,\;\,X \cap A =0$ und $(X \oplus A)/X$ rein in $B/X$.
Wir m"ussen zeigen, da"s $X=0$ ist.

Zun"achst ist $X \subset H(B)$, denn f"ur alle $n\geq 1$ gilt mit $\cy =\my^n$ und $\overline{B} =B/X,$ da"s nach Voraussetzung $\overline{A} \cap \cy \cdot \overline{B} = \cy \cdot \overline{A}$ ist, also $X \cap (A+\cy B)=X\cap \cy B:\; x \in X \cap (A+ \cy B) \Rightarrow x-a \in \cy B,\;\bar{a} \in \overline{A} \cap \cy \cdot \overline{B},\;\bar{a} \in \cy \cdot \overline{A},\;a \in \cy A$ wegen $X \cap A=0$, also $x \in \cy B.$ Weil nach Voraussetzung $B/A$ radikalvoll, also $A+\cy B=B$ ist, folgt $X \subset \cy B =\my^n B.$

Aus $X \subset H(B)$ und $X \cap H(A)=0$ folgt mit der dritten Bedingung $X=0.$
\end{proof}  

\begin{corollary}
Ist $M$ separiert und $M$  rein in $\hat{M}$, so gilt:
\begin{enumerate}
\item[(a)]
$M$ ist rein-wesentlich in $\hat{M}$.
\item[(b)]
$M$ ist totalreduziert, d.h. f"ur jeden endlich erzeugten $R$-Modul $X$ ist $X \otimes_R M$ reduziert.
\item[(c)]
Ass$(\hat{M}) \,\subset $ Koatt$(M).$
\end{enumerate}
\end{corollary}

\begin{proof}
\begin{enumerate}
\item[(a)]
$\;A=M$ und $B=\hat{M}$ erf"ullen die Voraussetzungen von (1.9), denn nat"urlich ist $\hat{M}/M$ radikalvoll und sogar $H(\hat{M})=0.$
\item[(b)]
Mit einer exakten Folge $\;R^n \stackrel{\alpha}{\longrightarrow} R^m \stackrel{\beta}{\longrightarrow} X \rightarrow 0\;$ ist auch $R^n \otimes_R \hat{M} \rightarrow R^m \otimes_R \hat{M} \rightarrow X \otimes_R \hat{M} \to 0\,$ exakt, also Kern$(\beta \otimes 1) $ 
als Faktormodul von $R^n \otimes_R \hat{M}$ vollst"andig. Nach (1.4) ist deshalb $X \otimes_R \hat{M}$ reduziert, also auch der Untermodul $X \otimes_R M.$
\item[(c)]
Die rein-injektive H"ulle $M \subset N$ l"a"st sich, weil jetzt $M \subset \hat{M}$ rein-wesentlich ist, zu einem Monomorphismus $\hat{M} \to N$ hochheben, und wieder mit (1.5a) folgt Ass$(\hat{M}) \subset $ Ass$(N) \subset $ Koatt$(M).$
\end{enumerate}
\end{proof}

\begin{remark}
{\rm Ohne die Reinheit von $M$ in $\hat{M}$ gilt (b) nicht einmal f"ur $X = R/\py$:
Nach Griffith \cite[p.323]{4} gibt es "uber jedem 2-dimensionalen, abz"ahlbaren, regul"aren Ring $R$ einen flachen Untermodul $M \subset R^{({\Bbb N})}$ und ein Primideal $\py \subsetneq \my$, so da"s $M/\py M$ den Quotientenk"orper von $R/\py$ enth"alt, also nicht reduziert ist.}
\end{remark}

{\it Frage 1}\quad F"ur welche separierten $R$-Moduln $M$ gilt Ass$(\hat{M})\subset $ Koatt$(M)$?

\section{Totalseparierte Moduln}
\setcounter{theorem}{0}

Ein $R$-Modul $M$ hei"se {\it totalsepariert}, wenn $X \otimes_R M$ separiert ist f"ur jeden endlich erzeugten $R$-Modul $X$.
Nat"urlich ist jeder endlich erzeugte $R$-Modul totalsepariert, allgemeiner jeder koatomare $R$-Modul $M$, denn dann ist $\my^eM$ endlich erzeugt f"ur ein $e\geq 1$, also auch $X \otimes_R M$ koatomar, insbesondere separiert.
"Uber einem diskreten Bewertungsring $R$ ist sogar jeder separierte $R$-Modul $M$ bereits totalsepariert, denn $X$ ist von der Form
$X \cong R/\ay_1\times \cdots \times  R/\ay_n$, und weil offenbar alle $R/\ay_i \otimes_R M \cong M/\ay_i M$ separiert sind, ist es auch $X \otimes_R M.$

\begin{lemma}
F"ur einen $R$-Modul $M$ sind "aquivalent:
\begin{enumerate}
\item[(i)]
$M$ ist totalsepariert.
\item[(ii)]
Die kanonische Abbildung $M \to \prod\limits_{n\geq 1} M/\my^n M\;$ ist ein reiner Monomorphismus.
\item[(iii)]
Die rein-injektive H"ulle $N$ von $M$ ist separiert.
\end{enumerate}
\end{lemma}

\begin{proof}
$(i \to ii)$\quad Wir m"ussen zeigen, da"s f"ur jeden endlich erzeugten $R$-Modul $X$ die kanonische Abbildung $X \otimes_R M \to X \otimes_R (\prod\limits_{n\geq 1} M/\my^n M)\;$ injektiv ist.
Nach Voraussetzung ist nun $X \otimes_R M$ separiert, d.h. die Abbildung $X \otimes_R M \to \prod\limits_{n\geq 1} (X \otimes_R M)\,\otimes_R R/\my^n\,\cong\,\prod\limits_{n\geq 1} (X \otimes_R M/\my^n M)\,$ injektiv, und weil $X \otimes_R - $ mit beliebigen Produkten vertauscht, ist das die Behauptung.

$(ii \to iii)$\quad Ein beliebiges Produkt von reinen Epimorphismen $\beta_i: B_i\to C_i$ ist wieder rein, denn f"ur jeden endlich erzeugten $R$-Modul $X$ sind alle Hom$_R(X,B_i) \to $ Hom$_R(X,C_i)$ surjektiv f"ur $i \in I$, also auch Hom$_R(X,\prod\limits_{i \in I} B_i) \to $ Hom$_R(X,\prod\limits_{i \in I} C_i).$

Ist nun $M \subset N$ eine rein-injektive H"ulle, sind alle $\alpha_n: M/\my^n M \to N/\my^n N$ reine Monomorphismen, also im kommutativen Diagramm

\vspace*{1cm}
\begin{center}
\begin{picture}(2200,700)
\setsqparms[0`1`1`1;1000`700]
\putsquare(0,0)[M`N`\prod\limits_{n\geq 1} M/\my^n M`\prod\limits_{n\geq 1} N/\my^n N;\subset`f`g`\alpha]
\end{picture}
\end{center}

\vspace{1cm}
     
auch $\alpha = \prod \alpha_n$ ein reiner Monomorphismus. Weil die kanonische Abbildung $f$ nach Voraussetzung rein ist, folgt aus der Definition von rein-wesentlich, da"s $g$ injektiv, d.h. Kern$\,g = H(N)=0$ ist.

$(iii\to i)$\quad
$N$ ist separiert, also nach (1.5c) direkter Summand von $\prod\limits_{i \in I} Y_i,$
wobei alle $Y_i$ von endlicher L"ange sind. Klar sind alle $Y_i$ totalsepariert, also auch $\prod\limits_{i \in I} Y_i$, also auch der reine Untermodul $M$. 
\end{proof}

\begin{corollary}
F"ur einen separierten $R$-Modul $M$ sind "aquivalent:
\begin{enumerate}
\item[(i)]
$M \subset \hat{M}\;$ ist eine rein-injektive H"ulle.
\item[(ii)]
$M$ ist totalsepariert und $N/M$ radikalvoll.
\end{enumerate}
\end{corollary}

\begin{proof}
$(i \to ii)$\quad Stets ist $\hat{M}/M$ radikalvoll, und weil $\hat{M}$ rein-injektiv und separiert, also nach (1.5c) totalsepariert ist, ist das auch der nach Voraussetzung reine Untermodul $M$.

$(ii \to i)$\quad
Aus der zweiten Bedingung folgt nach (1.6) $\hat{M} \cong N/H(N)$, aus der ersten nach (2.1) $H(N)=0.$
\end{proof}

\begin{remark}
{\rm F"ur eine rein-wesentliche Erweiterung $A \subset B$ folgt unmittelbar aus der Definition, da"s jeder endlich erzeugte reine Untermodul von $B/A$ bereits Null ist.
Speziell "uber einem diskreten Bewertungsring $R$ sieht man mit Hilfe eines Basis-Untermoduls, da"s dann $B/A$ sogar radikalvoll ist,
und weil jeder separierte $R$-Modul nach der Einleitung bereits totalsepariert ist, erh"alt man:}

"Uber einem diskreten Bewertungsring $R$ gilt f"ur jeden separierten $R$-Modul $M$, da"s $M \subset \hat{M}$ eine rein-injektive H"ulle ist.
\end{remark}

\begin{theorem}
Sei $M$ separiert und $\hat{M}$ als $R$-Modul flach. Dann sind "aquivalent:
\begin{enumerate}
\item[(i)]
$M$ ist totalsepariert.
\item[(ii)]
F"ur jedes Primideal $\py$ von $R$ ist $M/\py M$ separiert.
\item[(iii)]
$M$ ist rein in $\hat{M}.$

Falls $R$ vollst"andig, ist das weiter "aquivalent mit
\item[(iv)]
$M$ ist ein Mittag-Leffler-Modul.
\end{enumerate}
\end{theorem}

\begin{proof}
$(i \to ii)$\quad
ist klar, weil $X \otimes_R M$ speziell f"ur $X=R/\py$ separiert ist.

$(ii \to iii)$\quad F"ur jedes Ideal $\ay$ von $R$ gilt: Ist $M/\ay M$ separiert, folgt $M \cap \ay \hat{M} = \ay M$, denn mit $P=\prod\limits_{n\geq 1} M/\my^n M$ ist die kanonische Abbildung $M/\ay M \to \prod\limits_{n\geq 1} (M/\ay M \otimes_R R/\my^n)\,\cong P/\ay P\;$ injektiv, also sogar $M \cap \ay P = \ay M.$

Unter unseren Voraussetzungen ist deshalb Tor$_1^R(R/\py,\hat{M}/M)=0$ f"ur 
jedes Primideal $\py$, also sogar Tor$_1^R(X,\hat{M}/M)=0$ f"ur jeden endlich erzeugten $R$-Modul $X$, und das bedeutet, da"s $\hat{M}/M$ flach, also $M$ rein in $\hat{M}$ ist.

$(iii\to i)\quad$ Jeder vollst"andige flache $R$-Modul ist wieder nach Jensen \cite[Proposition 4]{5} rein-injektiv, insbesondere nach (1.5c) totalsepariert. Mit $\hat{M}$ ist also auch der nach Voraussetzung reine Untermodul $M$ totalsepariert.

$(iv \to i)$\quad gilt auch ohne die Vollst"andigkeit von $R$. Ein $R$-Modul $M$ hei"st nach Raynaud und Gruson \cite[2.1.5]{7} {\it Mittag-Leffler-Modul},
wenn f"ur jede Familie $(Q_i|i \in I)$ von $R$-Moduln die kanonische Abbildung $M \otimes_R (\prod Q_i) \to \prod (M \otimes_R Q_i)\,$ injektiv ist.
$M$ ist dann separiert, denn zu jedem $ x\in H(M)$ gibt es nach \cite[2.2.1]{7} einen rein-projektiven
Untermodul $U$ von $M$, so da"s $x \in U$ und $U$ rein in $M$ ist:
Aus dem ersten folgt $H(U)=0$, aus dem zweiten $x \in U \cap H(M)=H(U),$ also $x=0$.

$M$ ist sogar totalsepariert, denn f"ur jeden endlich erzeugten $R$-Modul $X$ ist $X \otimes_R M$ wieder ein Mittag-Leffler-Modul: Im kommutativen Diagramm

\vspace*{2cm}

\begin{center}
\begin{picture}(2200,700)
\settriparms[1`0`1;500]
\putCtriangle(0,0)[`M\otimes_R (X \otimes_R \prod Q_i)`;\cong``\beta\;\cong]
\setsqparms[1`0`1`1;1500`1000]
\putsquare(900,0)[(X \otimes_R M) \otimes_R(\prod Q_i)`\prod ((X \otimes_R M)\otimes_R Q_i)`M \otimes_R (\prod (X \otimes_R Q_i))`\prod (M \otimes_R (X \otimes_R Q_i));\gamma``\cong`\alpha]
\end{picture}
\end{center}

\vspace{1cm}
     
ist nach Voraussetzung $\alpha$ injektiv, ebenso $\beta$ wegen $X$ endlich erzeugt, also auch $\gamma$ wie behauptet.

$(ii \to iv)$\quad Nach dem bereits Bewiesenen ist auch $M$ flach, also $M$ nach \cite[2.5.3]{7} wegen der Vollst"andigkeit von $R$ ein Mittag-Leffler-Modul.
\end{proof}

\begin{corollary}
Ist $\dim (R) \leq 1$, so ist jeder separierte flache $R$-Modul bereits totalsepariert.
\end{corollary}

\begin{proof}
Mit Punkt $(ii)$ des Satzes: Nat"urlich ist $M/\my M$ separiert, und bei $\py \neq \my$ l"a"st sich $R/\py$ in $R$ einbetten, also auch $M/\py M$ in $M$, so da"s $M/\py M$ separiert ist.
\end{proof}

\begin{remark}
{\rm Die Implikation $(i \to iv)$ im Satz gilt nicht mehr, sobald $R$ unvollst"andig ist: Aus der rein-exakten Folge $0 \to R \to \hat{R} \to \hat{R}/R \to 0\,$ folgt mit $C=\hat{R}/R$, da"s $0 \to \hat{R} \to \hat{R} \otimes_R \hat{R} \to \hat{R} \otimes_R C \to 0\,$ zerf"allt und $\hat{R}\otimes_R C$ radikalvoll $\neq 0$ ist, also $\hat{R} \otimes_R \hat{R}$ einen radikalvollen direkten Summanden $\neq 0$ besitzt.
Damit ist $\hat{R} \otimes_R \hat{R}$ kein Mittag-Leffler-Modul, also auch nicht $M=\hat{R}$, obwohl $M$ totalsepariert ist.}
\end{remark}

\begin{remark}
{\rm Die Implikation $(iii \to i)$ im Satz gilt auch dann, wenn $M$ separiert und $M/\my M$ endlich erzeugt ist:
Weil dann $M/\my M \to M^{00}/\my \cdot M^{00}$, also auch $M/\my M \to N/\my N$ ein Isomorphismus ist, ist $N/M$ radikalvoll und nach (1.6) $\hat{M} \cong N/H(N)$, also $\hat{M}$ nach (1.5c) totalsepariert. Wie in (2.4, $iii \to i$) folgt die Behauptung.}
\end{remark}

Der in (1.11) angegebene Modul aus \cite{4} ist ein erstes Beispiel daf"ur, da"s ein separierter flacher $R$-Modul $A$ nicht rein in $\hat{A}$ sein mu"s.
Wir wollen jetzt eine ganze Reihe von solchen Beispielen angeben, die auf folgendem Prinzip beruhen: Ist $M$ separiert und $M \subset A \subset \hat{M}$ ein Zwischenmodul, so folgt aus dem kommutativen Diagramm

\vspace*{1cm}

\begin{center}
\begin{picture}(2200,700)
\settriparms[1`1`-1;700]
\putVtriangle(0,0)[M/\my^n M`\hat{M} /\my^n \hat{M}`A/\my^n A;\cong``]
\end{picture}
\end{center}

\vspace{1cm}
     
da"s $A$ genau dann ein dichter Unterraum von $\hat{M}$, d.h. $\hat{A} = \hat{M}$ ist, wenn $A/M$ radikalvoll ist.
Aber $A$ mu"s keineswegs rein in $\hat{M}$ sein:

\begin{beispiel}
Sei $R$ ein unvollst"andiger Integrit"atsring mit $\dim(R)\geq 2$, so da"s $R/\ay$ vollst"andig ist f"ur alle Ideale $\ay \neq 0.$
Dann gibt es einen separierten flachen $R$-Modul $A$, so da"s $A/\my A$ einfach ist, aber $A$ {\rm nicht} rein in $\hat{A}.$
\end{beispiel}

\begin{proof}
Nach Rotthaus \cite[1.4]{8} gibt es zu jedem $n\geq 2$ einen Integrit"atsring $R$ mit $\dim (R)=n$ und der gew"unschten Eigenschaft. Nach \cite[Beispiel 2.4]{12} ist dann Koass$_R(\hat{R}) =\{0,\my\},$ also Koass$_R(\hat{R}/R)=\{0\},$
d.h. $\hat{R}/R \cong K^{(I)}$ mit $K=$ Quot$(R)$ und $I \neq \emptyset.$
Mit irgendeinem $0 \neq \py \subsetneq \my$ ist dann $X=R_{\py}$ ein flacher, radikalvoller Untermodul von $K$, aber nicht rein in $K$.
W"ahlt man Zwischenmoduln $R \subset A \subset B \subset \hat{R}$ mit $B/R \cong K,\;\,A/R \cong X,\,$ so ist $A$ separiert und flach, $A/\my A$ einfach, $A$ nicht rein in $\hat{R} =\hat{A}.$
\end{proof}

\begin{remark}
{\rm Ist $M$ wie in der Voraussetzung zu (2.4) separiert und $\hat{M}$ als $R$-Modul flach, mu"s $M$ selbst nicht flach sein (so da"s insbesondere $M$ nicht rein in $\hat{M}$ ist).
Ein Beispiel dazu geben Bartijn und Strooker in \cite[Example 3.11a]{1}, weitere folgen mit unserer Konstruktion in (2.8):
Ist $R$ ein Integrit"atsring mit $\dim (R)\geq 3$, so gibt es einen Untermodul $X$ von $K$, der radikalvoll, aber nicht flach ist.
Ist zus"atzlich Koass$_R(\hat{R}) = \{0,\my\}$ und $R \subset A \subset B \subset \hat{R}$ wie in (2.8), wird $\hat{A} =\hat{R}$ flach, obwohl $A/\my A$ einfach und $A$ {\rm nicht} flach ist. }
\end{remark}

In unseren Beweisen wurde mehrfach ben"utzt, da"s ein vollst"andiger, flacher $R$-Modul rein-injektiv, also totalsepariert ist.

{\it Frage 2} \quad Ist jeder vollst"andige $R$-Modul $M$ totalsepariert?

\section{Direkte Summen von koatomaren Moduln}

\setcounter{theorem}{0}

F"ur jeden koatomaren $R$-Modul $M$ gilt Ass$(\hat{M}) = $ Koatt$(M)$: Weil $M$ totalsepariert, also rein in $\hat{M}$ ist, gilt Ass$(\hat{M}) \subset $ Koatt$(M)$, und wegen Koatt$(M) = $ Ass$(M)$ folgt die Behauptung. Wir wollen
mit einigem Aufwand dieselbe Formel f"ur beliebige direkte Summen von koatomaren $R$-Moduln beweisen.

\begin{lemma}
Sei $M=\coprod\limits_{i \in I} M_i$ separiert und seien alle $M_i$ rein in $\hat{M}_i.$
Dann ist auch $M$ rein in $\hat{M}$.
\end{lemma}

\begin{proof}
Nach Simon \cite[p.244]{9} kann man $\hat{M}$ als Modul zwischen $\coprod \hat{M}_i$ und $\prod \hat{M}_i$ auffassen, n"amlich
$$(\ast)\qquad \hat{M}\;=\;\{ (x_i) \in \prod \hat{M}_i|\mbox{ F"ur jedes } n\geq 1 \mbox{ gilt: Fast alle } x_i \in \my^n \cdot \hat{M}_i\}.$$
Stets ist $\coprod \hat{M}_i$ rein in $\prod \hat{M}_i$, also erst recht in $\hat{M}$, und weil nach Voraussetzung alle $M_i$ rein in $\hat{M}_i$, also auch $\coprod M_i$ rein in $\coprod \hat{M}_i$ ist, folgt die Behauptung.
\end{proof} 

\begin{corollary}
Ist in $M=\coprod M_i$ jedes $M_i$ vollst"andig und Untermodul eines flachen $R$-Moduls, gilt Ass$(\hat{M})=$ Koatt$(M)$.
\end{corollary}

\begin{proof}
Aus $M$ rein in $\hat{M}$ folgt Ass$(\hat{M}) \subset $ Koatt$(M)$, und weil jedes $M_i$, also auch $\prod M_i$, also nach $(\ast)$ auch $\hat{M}$ Untermodul eines flachen $R$-Moduls ist, ist Ass$(\hat{M}) \subset $ Ass$(R)$ endlich und es folgt Gleichheit.
\end{proof}

\begin{proposition}
Sei $R$ ein vollst"andiger Integrit"atsring und sei $M= \coprod\limits_{i=1}^\infty R/\ay_i$ ein treuer $R$-Modul. Dann folgt $0 \in$ Ass$(\hat{M}).$
\end{proposition}

\begin{proof}
Aus $\bigcap\limits_{i=1}^\infty \ay_i=0$ folgt mit $\by_i=\ay_1 \cap \cdots \cap \ay_i$, da"s $\by_1 \supset \by_2\supset \cdots $ ist und $\bigcap\limits_{i=1}^\infty \by_i=0.$

Falls $R$ ein K"orper, ist nichts zu zeigen. Falls $R$ kein K"orper, folgt mit irgendeinem $0\neq r \in \my$ f"ur jedes $n\geq 1$, da"s $\by_1:(r^n)\supset \by_2:(r^n)\supset \cdots$ ist und
$\bigcap\limits_{i=1}^\infty \by_i:(r^n)=0$, also nach dem Theorem von Chevalley \cite[Chap. III, \S2, Prop. 8]{2} ein $j_n \geq 1$ existiert mit $\by_{j_n}:(r^n)\subset \my^n.\,$
Wir wollen gleich $j_1<j_2<j_3<\cdots$ annehmen, und mit $x_i := \overline{r^n} \in R/\ay_i$ f"ur alle $j_{n-1} <i \leq j_n$ leistet dann
$$x\;=\;(x_i)\;=\;(\bar{r},...,\bar{r}, \overline{r^2},...,\overline{r^2},\overline{r^3},...)\;\in \;\prod_{i=1}^\infty R/\ay_i$$
das Gew"unschte: F"ur jedes $n\geq 1$ sind fast alle $x_i \in \my^n \cdot R/\ay_i$, so da"s nach $(\ast)$ folgt $x \in \hat{M}.$
Und weil Ann$_R(x) \subset \ay_1:(r^n)\cap \ay_2:(r^n)\cap \cdots \cap \ay_{j_n}:(r^n)=\by_{j_n}:(r^n)\subset \my^n\,$ ist f"ur alle $n\geq 1$, folgt Ann$_R(x)=0$, d.h. $0 \in $ Ass$(\hat{M})$ wie behauptet.
\end{proof}

\begin{theorem}
Seien in $M=\coprod\limits_{i \in I} M_i$ alle $M_i$ koatomar. Dann gilt Ass$(\hat{M}) = $ Koatt$(M).$
\end{theorem} 

\begin{proof}
Weil $M$ totalsepariert, also rein in $\hat{M}$ ist, gilt wieder Ass$(\hat{M}) \subset $ Koatt$(M)$. Die Umkehrung zeigen wir in drei Schritten:

{\it 1. Schritt}\quad Ist $R$ vollst"andig und sind in $M=\coprod\limits_{i=1}^\infty M_i$ alle $M_i$ koatomar, gilt Koatt$(M)\subset $ Ass$(\hat{M})$. Aus Koatt$(M_i)=$ Ass$(M_i)$ f"ur alle $i \in I$ folgt nach \cite[Folgerung 1.7]{14} Koatt$(\prod M_i)=$ Ass$(\prod M_i)$, so da"s zu
jedem $\py \in $ Koatt$(M)$ ein $y=(y_i) \in \prod M_i$ existiert mit $\py =$ Ann$_R(y).$
Weil $R/\py$ ein vollst"andiger Integrit"atsring und $\coprod R y_i$ als $R/\py$-Modul treu ist, gibt es nach  der Proposition  ein $x \in \prod R y_i$, so da"s Ann$_{R/\py}(x)=0$ ist und f"ur jedes $n\geq 1$ gilt: Fast alle $x_i$ liegen in $\overline{\my}^n\cdot R y_i$, also in $\my^n \cdot R y_i \subset \my^n \cdot M_i$.
Damit ist $x \in \hat{M}$ und Ann$_R(x)=\py$ wie gew"unscht.

{\it 2. Schritt} \quad Ist $R$ vollst"andig und sind in $M=\coprod\limits_{i \in I} M_i$ alle $M_i$ koatomar, gilt  Koatt$(M) \subset $ Ass$(\hat{M})$. Der abz"ahlbare Fall ist der erste Schritt. Sei also $I$ "uberabz"ahlbar und $\py \in $ Koatt$(M)$, d.h. $\py = $ Ann$_R(M[\py])\,=\bigcap\limits_{i \in I} $ Ann$_R(M_i[\py])$.
Nach \cite[Lemma 2.1]{11} gibt es eine abz"ahlbare Menge mit demselben Durchschnitt, d.h. paarweise verschiedene $i_1,i_2,i_3,...\,\in I$ mit $\py =\bigcap\limits_{m=1}^\infty $ Ann$_R(M_{i_m}[\py]).\;\;U=\coprod\limits_{m=1}^\infty M_{i_m}$ ist direkter Summand von $M$ und $\py \in $ Koatt$(U)$, also nach dem 1. Schritt $\py \in $ Ass$(\hat{U}) \subset $ Ass$(\hat{M})$ wie gew"unscht.

{\it 3. Schritt} \quad Ist schlie"slich $R$ beliebig und $M$ wie oben,
sind in $\hat{R}\otimes_R M \cong \coprod\limits_{i \in I} \hat{R} \otimes_R M_i\,$ alle Summanden  als $\hat{R}$-Moduln koatomar, so da"s nach dem 2. Schritt Koatt$_{\hat{R}}(\hat{R} \otimes_R M) \subset $ Ass$_{\hat{R}}(\widehat{\hat{R} \otimes_R M})$ gilt, wegen $\widehat{\hat{R} \otimes_R M} \cong \hat{M}$ nach \cite[p.17]{10} also Koatt$_{\hat{R}}(\hat{R} \otimes_R M)\subset $ Ass$_{\hat{R}} (\hat{M}).$
Sei nun $\py \in $ Koatt$(M):$ Mit $\py = $ Ann$_R(U)$ und $A= $ Ann$_{\hat{R}}(\hat{R} \otimes_R U)$ wird $R/\py \to \hat{R} /A$ injektiv, so da"s es ein $A \subset Q \in $ Spec$(\hat{R})$ gibt mit $Q \cap R = \py$,
dazu ein Primideal $A \subset P \subset Q$, das minimal "uber $A$ ist, und dann ist $P \in $ Koatt$_{\hat{R}}(\hat{R} \otimes_R U) \subset $ Koatt$_{\hat{R}}(\hat{R} \otimes_R M)$ sowie $P \cap R = \py.$
Nach dem ersten Teil folgt $P \in $ Ass$_{\hat{R}}(\hat{M})$, also $\py \in $ Ass$(\hat{M})$ wie gew"unscht.
\end{proof}

\begin{beispiel}
Ist $M=\coprod\limits_{i=1}^\infty R/\my^i$, so gilt Ass$(\hat{M}) = \{\my\}\cup $ Ass$(R).$
\end{beispiel}

\begin{proof}
Nach \cite[p.1984]{14} ist Koatt$(\coprod\limits_{i=1}^\infty R/\my^i)\,=\{\my\}\cup $ Ass$(R).$
\end{proof}

\begin{beispiel}
Ist $M=\coprod\limits_{i=1}^\infty (R/\my^i)^0$, so gilt Ass$(\hat{M}) = $ Spec$(R).$
\end{beispiel}

\begin{proof}
F"ur jedes Primideal $\py$ gilt, da"s $M[\py] \cong \coprod\limits_{i=1}^\infty (R/\my^i+\py)^0$ ist, also Ann$_R(M[\py])=\bigcap\limits_{i=1}^\infty (\my^i+\py)=\py$, d.h. $\py \in $ Koatt$(M).$
\end{proof}

{\it Frage 3}\quad F"ur welche separierten $R$-Moduln $M$ gilt Koatt$(M) \subset $ Ass$(\hat{M})$?

\vspace{0.5cm}

\begin{remark}
{\rm F"ur einen beliebigen $R$-Modul $M$ betrachten wir die folgenden drei Bedingungen: (A) $\;$ Ass$(M)= $ Koatt$(M),\;$ (S) $\;\supp (M) = V($Ann$_R(M))\;$ und (B) $\;$ Ass$(M)$ besitzt eine endliche finale Teilmenge.
Weil (S) "aquivalent damit ist, da"s jeder minimale Primdivisor von Ann$_R(M)$ zu Ass$(M)$ geh"ort, gilt stets (A $\Rightarrow$ S) und (S $\Rightarrow$ B).
Erf"ullen aber alle vollst"andigen $R$-Moduln die Bedingung (B), so auch (A):
Aus $M$ vollst"andig und $\py \in $ Koatt$(M)$ folgt $M_1=M[\py]$ vollst"andig und $\py = $ Ann$_R(M_1)$, also nach (1.1) $\py = \bigcap $ Ass$(M_1)$ und wegen (B) $ \py \in $ Ass$(M_1) \subset $ Ass$(M).$

Simon stellte in \cite[p.244]{9} die bis heute unbeantwortete Frage, ob alle vollst"andigen $R$-Moduln die Bedingung (S) erf"ullen. Nach dem Vorhergehenden ist das "aquivalent mit der Frage, ob {\it alle} separierten $R$-Moduln $M$ die Bedingung Koatt$(M) \subset $ Ass$(\hat{M})$ erf"ullen.}
\end{remark}

\begin{remark}
{\rm Auch unser in \cite[p.197]{11} gestelltes Problem ist bis heute ungel"ost: Gilt f"ur jeden $R$-Modul $A$, da"s Koass$(A) $ eine endliche finale Teilmenge besitzt?

Dabei hei"st ein Primideal $\py$ {\it koassoziiert} zu $A$, wenn es einen artinschen Faktormodul $A/U$ gibt mit $\py = $ Ann$_R(A/U)$, und wegen Koass$(A)= $ Ass$(A^0)$ ist unsere Frage "aquivalent damit, ob alle dualen $R$-Moduln $M=A^0$ die Bedingung (B) in (3.7) erf"ullen.}
\end{remark}

\end{document}